\documentclass[11pt]{amsart}

\usepackage{amssymb, amsmath, amsthm}

 \newtheorem{thm}{Theorem}[section]
 \newtheorem{cor}[thm]{Corollary}
 \newtheorem{lem}[thm]{Lemma}
 \newtheorem{prop}[thm]{Proposition}

    \theoremstyle{definition}
 \newtheorem{defn}[thm]{Definition}
  \newtheorem{notn}[thm]{Notation}
  
  \newtheorem{example}[thm]{Example}
 \theoremstyle{remark}
 \newtheorem{rem}[thm]{Remark}
\numberwithin{equation}{section}
\DeclareMathOperator{\Pspec}{P.Spec}
\DeclareMathOperator{\spec}{Spec}

\DeclareMathOperator{\prim}{Prim}
\DeclareMathOperator{\Pprim}{P.Prim}

\begin{document}

\title[semiclassical limit]{A natural map from a quantized space onto its semiclassical limit and a multi-parameter Poisson Weyl algebra}

\author{Sei-Qwon Oh}

\address{Department of Mathematics, Chungnam National  University, 99 Daehak-ro,   Yuseong-gu, Daejeon 34134, Korea}
 \email{sqoh@cnu.ac.kr}

\thanks{The author is supported by National  Research Foundation of Korea Grant 2012-007347.}



\date{August 31, 2015.}


\begin{abstract}
A natural map from a quantized space onto its semiclassical limit is obtained. As an application, we see that an induced map by the natural map is a homeomorphism  from the spectrum of the multi-parameter quantized Weyl algebra onto the Poisson spectrum of its semiclassical limit.
\end{abstract}

\maketitle


\section*{Introduction}

A quantized coordinate ring $\mathcal{O}_q(V)$ of an affine variety $V$ is  informally a deformation of its classical coordinate ring $\mathcal{O}(V)$. Let us replace the parameter $q\in \mathcal{O}_q(V)$  by an indeterminate $t$. If a central element $h$ of $\mathcal{O}_t(V)$ is  non-zero, non-unit and non-zero-divisor and the factor algebra  $\mathcal{O}_t(V)/h\mathcal{O}_t(V)$ is commutative then $\mathcal{O}_t(V)/h\mathcal{O}_t(V)$ is a Poisson algebra with Poisson bracket defined by $\{\overline{a},\overline{b}\}=\overline{h^{-1}(ab-ba)}$. In such case, the Poisson algebra $\mathcal{O}_t(V)/h\mathcal{O}_t(V)$ is called a semiclassical limit of $\mathcal{O}_t(V)$ and $\mathcal{O}_t(V)$ is called a quantization of the Poisson algebra $\mathcal{O}_t(V)/h\mathcal{O}_t(V)$. Here we study a class $\mathcal{O}_q(V)$ of quantized algebras containing many cases such as  \cite{Good4}, \cite{GoLet1},  \cite{JoOh}, \cite{Oh7}, \cite{OhPaSh1}, \cite{OhPa}. Our aim of this paper is to construct a natural map from $\mathcal{O}_q(V)$ onto its semiclassical limit.  As an application, we prove that this natural map induces a homeomorphism between the spectrum of quantized Weyl algebra and the Poisson spectrum of its semiclassical limit.

The multiparameter quantized Weyl algebra $R_n^{Q,\Lambda}$ was constructed by G.~Maltsiniotis \cite{Mal}. The spectrum  $\spec R_n^{Q,\Lambda}$ was described by M.~Akhavizadegan and David~A. Jordan \cite{AkJo} and K. Goodearl and E. Letzter
\cite{GoLet1} proved  that $R_n^{Q,\Lambda}$ satisfies the Dixmier-Moeglin equivalence.
 Here we find a semiclassical limit $A_1$ of $R_n^{Q,\Lambda}$ and prove that $A_1$ satisfies the Poisson Dixmier-Moeglin equivalence. Moreover we construct a natural homeomorphism  from $\spec R_n^{Q,\Lambda}$ onto $\Pspec A_1$ such that its restriction to the primitive spectrum $\prim R_n^{Q,\Lambda}$  is also a homeomorphism onto the Poisson primitive spectrum $\Pprim A_1$ by applying the main result of this paper. This gives an affirmative answer for $R_n^{Q,\Lambda}$ for the conjecture
posted in \cite[9.1]{Good4}: {\it Assume that $k$ is an algebraically closed field of characteristic zero, and let $A$ be a generic quantized coordinate ring of an affine algebraic variety $V$ over $k$. Then $A$ should be a member of a flat family of $k$-algebras with semiclassical limit $\mathcal{O}(V)$, such that the primitive spectrum of $A$ is homeomorphic to the space of symplectic cores in $V$, with respect to the semiclassical limit Poisson structure. Further, there should be compatible homeomorphisms from the primitive spectrum of $A$ onto the Poisson primitive spectrum of $\mathcal{O}(V)$ and from  the  spectrum of $A$ onto the Poisson spectrum of $\mathcal{O}(V)$.}

In the first section, we see a method how to construct a natural map from a quantized space into its semi-classical limit which can  induce a homeomorphism and the ideas are illustrated by the quantized coordinate ring $\mathcal{O}_q(\Bbb C^2)$ of affine 2-space. In the second section, we construct a semiclassical limit $A_1$ of $R_n^{Q,\Lambda}$ called a multiparameter Poisson Weyl algebra. Then we find its Poisson spectrum and Poisson primitive spectrum by modifying the proofs given by  M.~Akhavizadegan and David~A. Jordan \cite{AkJo} and prove that $A_1$ satisfies the Poisson  Dixmier-Moeglin equivalence.

Assume throughout the paper that the base field is the complex number field $\Bbb C$ and that all algebras considered have  unity.


\begin{defn}
(1) A commutative algebra $A$ over  $\Bbb C$ is said to be  a {\it Poisson algebra} if there exists a bilinear product $\{-,-\}$ on $A$, called a {\it Poisson bracket}, such that $(A, \{-,-\})$ is a Lie algebra and $\{ab,c\}=a\{b,c\}+\{a,c\}b$ for all $a,b,c\in A$.

(2) An ideal $I$ of a Poisson algebra $A$ is said to be a {\it Poisson ideal} if $\{I,A\}\subseteq I$.  A Poisson ideal $P$ is said to be {\it Poisson prime} if, for all Poisson ideals $I$ and $J$, $IJ\subseteq P$ implies $I\subseteq P$ or $J\subseteq P$.  If $A$ is noetherian  then a Poisson prime ideal of $A$ is a prime ideal by \cite[Lemma 1.1(d)]{Good3}.

 For an ideal $I$ of a Poisson algebra $A$, the largest Poisson ideal $\mathcal{P}(I)$ contained in $I$ is called the {\it Poisson core} of $I$. If $I$ is prime then $\mathcal{P}(I)$ is prime by \cite[3.3.2]{Di}. A Poisson ideal $P$ of $A$ is said to be {\it Poisson primitive} if $P=\mathcal{P}(M)$ for some maximal ideal $M$ of $A$.
Note that Poisson primitive is Poisson prime.
\end{defn}

\begin{defn}
(1) Let $R$ be an algebra. The spectrum of $R$, denoted by $\spec R$, is the set of all prime ideals of $R$ equipped with the Zariski topology.
The primitive spectrum, denoted by $\prim R$, is the subspace of $\spec R$ consisting of all primitive ideals of $R$.

(2) Let $A$ be a Poisson algebra. The Poisson spectrum of $A$, denoted by $\Pspec A$, is the set of all Poisson prime ideals of $A$ equipped with the Zariski topology. The Poisson primitive spectrum of $A$, denoted by $\Pprim A$, is the subspace of $\Pspec A$ consisting of all Poisson primitive ideals of $A$.
If $A$ is noetherian then $\Pspec A$ is a subspace of $\spec A$ since Poisson prime is prime.
\end{defn}


\section{Natural map}

\begin{notn}\label{ASSUM}
Let  $t$ be an indeterminate.

(1) Assume that  ${\bf K}$ is an infinite subset of the set $\Bbb C\setminus\{0,1\}$.

(2) Assume that
$\Bbb F$ is a subring of the ring of regular functions on ${\bf K}\cup\{1\}$ containing $\Bbb C[t,t^{-1}]$. That is,
\begin{equation}\label{REG}
\Bbb C[t,t^{-1}]\subseteq \Bbb F\subseteq\{f/g\in\Bbb C(t)| f,g\in\Bbb C[t]\text{ such that } g(1)\neq0,g(\lambda)\neq0\ \forall\lambda\in {\bf K}\}.
\end{equation}

(3)
Let $\Bbb F\langle x_1,\ldots,x_n\rangle$ be the free $\Bbb F$-algebra on the set $\{x_1,\ldots, x_n\}$. A finite product ${\bf x}$  of $x_i$'s (repetitions allowed) is called a monomial. For each $i=1,\ldots,r$, let $f_i$ be an $\Bbb F$-linear combination of monomials
$$f_i=\sum_{\bf x}a^i_{\bf x}(t){\bf x},\ \ a^i_{\bf x}(t)\in\Bbb F.$$
Set $$A=\Bbb F\langle x_1,\ldots,x_n\rangle/I,$$ where $I$ is the ideal of $\Bbb F\langle x_1,\ldots,x_n\rangle$ generated by $f_1,\ldots,f_r$. That is, $A$ is the $\Bbb F$-algebra generated by $x_1,\ldots,x_n$ subject to the relations
$$f_1,\ldots,f_r.$$

Note that $A$ is also  a $\Bbb C$-algebra since $\Bbb C\subseteq \Bbb F$.

(4) Assume that  $t-1$ is a nonzero, nonunit and non-zero-divisor of $A$ such that the factor $A_1=A/(t-1)A$ is commutative. Denote by $\gamma_1$ the canonical homomorphism of $\Bbb C$-algebras
    $$\gamma_1:A\longrightarrow A_1.$$

Note, by \cite[III.5.4]{BrGo}, that the commutative $\Bbb C$-algebra $A_1$ is a Poisson algebra with Poisson bracket
$$\{\gamma_1(a),\gamma_1(b)\}=\gamma_1((t-1)^{-1}(ab-ba))$$
for all $a,b\in A$. The Poisson algebra $A_1$ is said to be a {\it semiclassical limit} of $A$.

(5)
For each $\lambda\in{\bf K}$, let $A_\lambda$ be the $\Bbb C$-algebra generated by $x_1,\ldots,x_n$ subject to the relations $$f_1|_{t=\lambda},\ldots,f_r|_{t=\lambda},$$ where
$f_i|_{t=\lambda}=\sum_{\bf x}a^i_{\bf x}(\lambda){\bf x}$.
Note that $a^i_{\bf x}(\lambda)$ is a well-defined element of $\Bbb C$ by (\ref{REG}) and that there exists a canonical homomorphism of $\Bbb C$-algebras
$$\gamma_\lambda:A\longrightarrow A_\lambda$$
such that $\gamma_\lambda(t)=\lambda, \gamma_\lambda(x_i)=x_i$ for $i=1,\ldots,n$.

(6) Set
$$\widehat{A}=\prod_{\lambda\in{\bf K}}A_\lambda.$$
Let $$\pi_\lambda:\widehat{A}\longrightarrow A_\lambda$$ be the canonical projection onto $A_\lambda$ for each $\lambda\in{\bf K}$ and let $\gamma$ be the homomorphism of $\Bbb C$-algebras
\begin{equation}\label{gamm}
\gamma:A\longrightarrow \widehat{A},\ \ \gamma(a)=(\gamma_\lambda(a))_{\lambda\in{\bf K}}.
\end{equation}
 Thus $\pi_\lambda\gamma=\gamma_\lambda$ for each $\lambda\in{\bf K}$.

Note that $\gamma(t-1)$ is an invertible element of $\widehat{A}$ since $1\notin{\bf K}$.

(7)
Assume that there exists an $\Bbb F$-basis $\{\xi_i|i\in I\}$ of $A$ such that $$\{\gamma_1(\xi_i)|i\in I\},\ \ \ \{\gamma_\lambda(\xi_i)|i\in I\}$$ are $\Bbb C$-bases of $A_1$ and $A_\lambda$, respectively, for each $\lambda\in {\bf K}$.

Hence every element $a\in A$ is expressed uniquely by
$$a=\sum_ia_i(t)\xi_i,\ \ a_i(t)\in\Bbb F$$
and, for each $\lambda\in {\bf K}$,
$$\begin{array}{c}\gamma_\lambda(a)=\sum_ia_i(\lambda)\gamma_\lambda(\xi_i),\ \ \
\gamma_1(a)=\sum_ia_i(1)\gamma_1(\xi_i).
\end{array}$$
Note that $a_i(\lambda)$ and $a_i(1)$ are well-defined elements of $\Bbb C$ by (\ref{REG}).
\end{notn}

\begin{lem}\label{MONO}
 The map  $\gamma$ in (\ref{gamm}) is a monomorphism of $\Bbb C$-algebras.
\end{lem}

\begin{proof}
For $a=\sum_ia_i(t)\xi_i\in A$ ($a_i(t)\in {\Bbb F }$), let $\gamma(a)=0$. Then
$$\sum_ia_i(\lambda)\gamma_\lambda(\xi_i)=\gamma_\lambda(a)=\pi_\lambda\gamma(a)=0$$ for all $\lambda\in{\bf K}$.
 Hence, for each $i\in I$,  $a_i(\lambda)=0$ for all $\lambda\in {\bf K}$ since $\{\gamma_\lambda(\xi_i)|i\in I\}$  is a $\Bbb C$-basis of $A_\lambda$. Thus $a_i(t)=0$ since ${\bf  K}$ is an infinite set and $a_i(t)\in \Bbb F$ has finite zeros.  It follows that $\gamma$ is a monomorphism.
\end{proof}

\begin{notn}
The map
$$\Gamma=\gamma_1\gamma^{-1}:\gamma(A)\longrightarrow A_1$$
is a homomorphism of $\Bbb C$-algebras by Lemma~\ref{MONO}.  Note that, for $a=\sum_ia_i(t)\xi_i\in A$,
\begin{equation}\label{AAA}
\Gamma(\gamma(a)))=\sum_ia_i(1)\gamma_1(\xi_i).
\end{equation}
In particular, $\Gamma(\gamma(\xi_i))=\gamma_1(\xi_i)$ for all $i\in I$ and thus $\Gamma$ is an epimorphism.
\end{notn}

\begin{thm}\label{homeomorphism}
Let  $I$ be an ideal of $\widehat{A}$. Then $\Gamma(I\cap \gamma(A))$ is a Poisson ideal of $A_1$.
\end{thm}

\begin{proof}
  Since $\Gamma$ is surjective, $\Gamma(I\cap \gamma(A))$ is an ideal of $A_1$.
It is enough to show that
$$\{x',y'\}\in \Gamma(I\cap \gamma(A))$$
for any $x'\in \Gamma(I\cap \gamma(A))$ and $y'\in A_1$.
There exist $a\in I\cap \gamma(A)$ and $b\in \gamma(A)$ such that $\gamma_1\gamma^{-1}(a)=x'$ and $\gamma_1\gamma^{-1}(b)=y'$.
Set  $\gamma^{-1}(a)=x$, $\gamma^{-1}(b)=y$.  Thus $x'=\gamma_1(x)$ and $y'=\gamma_1(y)$.
Since $xy-yx=(t-1)z$ for some $z\in A$ by Notation~\ref{ASSUM}(4) and
$$(t-1)z=xy-yx=\gamma^{-1}(a)\gamma^{-1}(b)-\gamma^{-1}(b)\gamma^{-1}(a)=\gamma^{-1}(ab-ba),$$
we have that
$$\gamma(t-1)\gamma(z)=\gamma((t-1)z)=ab-ba$$
  and thus $$\gamma(z)=[\gamma(t-1)]^{-1}(ab-ba)\in I\cap\gamma(A)$$ since $\gamma(t-1)$ is invertible in $\widehat{A}$.
  Hence
$$\{x',y'\}=\{\gamma_1(x),\gamma_1(y)\}=\gamma_1(z)\in \Gamma(I\cap \gamma(A)),$$
as claimed.
\end{proof}

\begin{cor}\label{INC}
 Let $X$ be a  subspace of the ideals of $\widehat{A}$ equipped with Zariski topology and let $Y$ be a topological subspace of $\Pspec A_1$. Suppose that the map
    $\varphi_\Gamma:X\longrightarrow Y$ defined by
    \begin{equation}\label{GAMM}
    \varphi_\Gamma(P)=\Gamma(P\cap \gamma(A))
      \end{equation}
      is a bijection  satisfying the following condition: For any $P,Q\in X$, there exist subsets $I_P, I_Q\subseteq \gamma(A)$ such that $I_P$, $I_Q$, $\Gamma(I_P)$ and $\Gamma(I_Q)$ generate $P$, $Q$, $\varphi_\Gamma(P)$ and $\varphi_\Gamma(Q)$, respectively, and that $I_P\subseteq I_Q$ if and only if $\Gamma(I_P)\subseteq\Gamma(I_Q)$.  Then $\varphi_\Gamma$ is a  homeomorphism from $X$ onto $Y$.
\end{cor}

\begin{proof}
Note that $\varphi_\Gamma(P)$ is a Poisson ideal for any $P\in X$ by Theorem~\ref{homeomorphism}.
Clearly $\varphi_\Gamma$  is  a bijection such that $\varphi_\Gamma$ and $\varphi_\Gamma^{-1}$ preserve inclusions. Hence $\varphi_\Gamma$ is a homeomorphism from $X$ onto $Y$ by \cite[Lemma 9.4(a)]{Good4}.
\end{proof}

In the following example, we see a method how to construct a homeomorphism from $\spec\mathcal{O}_q(\Bbb C^2)$ onto the Poisson spectrum of its semiclassical limit.

\begin{example}\label{affine}
Set ${\bf K}=\Bbb C\setminus(\{0,1\}\cup\{\text{roots of unity}\})$ and $\Bbb F=\Bbb C[t,t^{-1}]$. Let $A$ be the  $\Bbb F$-algebra generated by $x, y$ subject to the relation
$$xy=tyx.$$
Then, for each $\lambda\in{\bf K}$, $A_\lambda$ is the $\Bbb C$-algebra generated by $x,y$ subject to the relation
$$xy=\lambda yx$$ and thus
the set $\{x^iy^j|i,j=0,1,\ldots\}$ of all standard monomials  forms an $\Bbb F$-basis of  $A$ and a $\Bbb C$-basis of $A_\lambda$.

Observe that $A_1=A/(t-1)A$ is the Poisson $\Bbb C$-algebra $\Bbb C[x,y]$ with Poisson bracket
$$\{x,y\}=xy.$$
Note that the triple $({\bf K},\Bbb F, A)$ satisfies the conditions of Notation~\ref{ASSUM}. We retain the notations in Notation~\ref{ASSUM}.

Let $0\neq q\in\Bbb C$ be not a root of unity and let $\mathcal{O}_{q}(\Bbb C^2)$  be  the coordinate ring of quantized affine $2$-space. That is,  $\mathcal{O}_{q}(\Bbb C^2)$ is the $\Bbb C$-algebra generated by $x, y$ subject to the relation $$xy=qyx.$$ Note that $\{x^iy^j|i,j=0,1,\ldots\}$ is a $\Bbb C$-basis of $\mathcal{O}_q(\Bbb C^2)$.
For any $\lambda\in{\bf  K}$, $\lambda$ is not a root of unity and thus there exists  the $\Bbb C$-algebra $\mathcal{O}_\lambda(\Bbb C^2)$ that is defined by substituting $\lambda$ to $q$ in $\mathcal{O}_{q}(\Bbb C^2)$.
(Hence $q$ is a nonzero complex number that is not a root of unity as well as plays a role as a parameter taking values in ${\bf K}$.)
Note that $A_\lambda=\mathcal{O}_\lambda(\Bbb C^2)$ for each $\lambda\in{\bf K}$.

It is well-known, see \cite[2.1(ii)]{GoLet3}, that
\begin{equation}\label{Paffine}
\spec \mathcal{O}_{q}(\Bbb C^2)=\{ \langle 0\rangle,   \langle x\rangle, \langle y\rangle, \langle x-\mu,y\rangle, \langle x, y-\nu\rangle |\ \mu,\nu\in\Bbb C\}.
\end{equation}
Hence, for each $\lambda\in{\bf K}$,
$$\spec \mathcal{O}_{\lambda}(\Bbb C^2)=\{ \langle 0\rangle,   \langle x\rangle, \langle y\rangle, \langle x-\mu,y\rangle, \langle x, y-\nu\rangle |\  \mu,\nu\in\Bbb C\}.$$

Define a map
\begin{equation}\label{AFF}
\widehat{}\ : \mathcal{O}_{q}(\Bbb C^2)\longrightarrow \widehat{A}=\prod_{\lambda\in{\bf K}}A_\lambda,\ \ \widehat{f(q)}=(f(\lambda))_{\lambda\in{\bf K}}.
\end{equation}
(Here $q$ is considered as a parameter taking values in ${\bf K}$.)
Hence, for each $\lambda\in{\bf K}$,
$$\pi_\lambda(\widehat{0})=0, \ \pi_\lambda(\widehat{x})=x, \ \pi_\lambda(\widehat{y})=y,\ \pi_\lambda(\widehat{x-\mu})=x-\mu,\ \pi_\lambda(\widehat{y-\nu})=y-\nu.$$
Note that
$$\gamma(0)=\widehat{0}, \ \ \ \gamma(x)=\widehat{x},\ \ \  \gamma(y)=\widehat{y},\ \ \ \gamma(x-\mu)=\widehat{x-\mu},\ \ \ \gamma(y-\nu)=\widehat{y-\nu}.$$

Let $X$ be the set consisting of the following ideals of $\widehat{A}$
$$X=\{ \langle\widehat{0}\rangle,   \langle \widehat{x}\rangle, \langle \widehat{y}\rangle, \langle \widehat{x-\mu},\widehat{y}\rangle, \langle \widehat{x}, \widehat{y-\nu}\rangle |\  \mu,\nu\in\Bbb C\}.$$
The map (\ref{AFF}) is an injective map preserving inclusions since $q\in{\bf K}$ and thus  (\ref{AFF}) induces a canonical homeomorphism $\widehat{\varphi}$ from $\spec\mathcal{O}_q(\Bbb C^2)$ onto $X$. Note that all ideals in $X$ are generated by elements of $\gamma(A)$,
 that
$$\varphi_\Gamma(X)=\{ \langle 0\rangle,   \langle x\rangle, \langle y\rangle, \langle x-\mu,y\rangle, \langle x, y-\nu\rangle |\  \mu,\nu\in\Bbb C\}$$
is equal to $\Pspec A_1$ by \cite[9.5]{Good4} and that $\varphi_\Gamma$ is a homeomorphism from $X$ onto $\varphi_\Gamma(X)$
by Corollary~\ref{INC}. Hence the composition $\varphi_\Gamma\circ \widehat{\varphi}$ \ is a homeomorphism from $\spec \mathcal{O}_q(\Bbb C^2)$ onto $\Pspec A_1$.
\end{example}

\section{Multi-parameter Poisson Weyl algebras}

\begin{defn}\label{Weyl} (See \cite{Mal}.)
Let $Q=(q_1,\ldots, q_n)$ be an $n$-tuple of elements of $\Bbb C\setminus\{0,1\}$ and let $\Lambda=(\lambda_{ij})$ be a multiplicative antisymmetric $n\times n$-matrix over $\Bbb C\setminus\{0\}$, namely $\lambda_{ij}=\lambda_{ji}^{-1}$ for $i\neq j$ and $\lambda_{ii}=1$. The multi-parameter quantized Weyl algebra $R_n^{Q,\Lambda}$ is  the $\Bbb C$-algebra generated by $y_1,x_1,\ldots, y_n,x_n$ subject to the relations
\begin{equation}\label{WA}
\begin{aligned}
y_jy_i&={\lambda}_{ji}y_iy_j&&  1\leq i<j\leq n\\
y_jx_i&={\lambda}_{ij}x_iy_j&&1\leq i<j\leq n\\
x_jy_i&={q}_i{\lambda}_{ij}y_ix_j&&1\leq i<j\leq n\\
x_jx_i&={q}_{i}^{-1}{\lambda}_{ij}^{-1}x_ix_j&&1\leq i<j\leq n\\
x_iy_i-{q}_iy_ix_i&=1+\sum_{k=1}^{i-1}(q_k-1)y_kx_k&\qquad& 1\leq i\leq n.
\end{aligned}
\end{equation}
\end{defn}

\begin{notn}
(1) Denote by $G(Q,\Lambda)$ the multiplicative subgroup of $\Bbb C\setminus\{0\}$ generated by  all $q_i$ and $\lambda_{ij}$, $1\leq i,j\leq n$.

(2) {\it Assume throughout the section that $G(Q,\Lambda)$ is torsion free.} Hence $G(Q,\Lambda)$ is a free abelian group with finite rank, say $r$. Fix a basis $\{\eta_1,\ldots, \eta_r\}$ of  $G(Q,\Lambda)$,  a $\Bbb Q$-linearly independent subset $\{\mu_1,\ldots, \mu_r\}$ of $\Bbb C$ and a non-root of unity $q\in\Bbb C\setminus\{0,1\}$.

(3)  For each $i=1,\ldots, r$, there exists a unique element $e_i=a_it^2+b_it+c_i\in\Bbb C[t]$ such that
\begin{equation}\label{WC}
e_i(q)=\eta_i,\  e_i(1)=1,\  e_i'(1)=\mu_i,
\end{equation}
 where $e_i'$ is the formal derivative of $e_i$, since the determinant of the matrix $\begin{pmatrix} q^2&q&1\\1&1&1\\2&1&0\end{pmatrix}$ is nonzero.

(3) Let $\Bbb F$ be the ring
$$\Bbb F=\Bbb C[t,t^{-1}][e_1^{-1},\ldots,e_r^{-1}].$$
and let $U(\Bbb F)$ be the unit group of $\Bbb F$.
For any $c\in G(Q,\Lambda)$, there exists a unique
element $(s_1,\ldots, s_r)\in\Bbb Z^r$ such that $c=\eta_1^{s_1}\eta_2^{s_2}\cdots\eta_r^{s_r}$. Define  a  homomorphism of groups
\begin{equation}\label{WB}
\widetilde{} \ :G(Q,\Lambda)\longrightarrow U(\Bbb F),\ \ \widetilde{c}=e_1^{s_1}e_2^{s_2}\cdots e_r^{s_r}.
\end{equation}
If $\widetilde{c}=1$ then $c=\widetilde{c}(q)=1$ and thus $\widetilde{}\ $\  is a monomorphism of groups such that $\widetilde{\eta_i}=e_i$ for all $i=1,\ldots,r$.
\end{notn}

\begin{defn}
 Set
$$\widetilde{Q}=(\widetilde{q}_1,\ldots, \widetilde{q}_n),\ \ \ \widetilde{\Lambda}=(\widetilde{\lambda}_{ij}).$$
Note that $\widetilde{q}_i\neq1$ and that $\widetilde{\Lambda}$ is a multiplicative antisymmetric $n\times n$-matrix.
 Define  $A_n^{\widetilde{Q},\widetilde{\Lambda}}$ to be the $\Bbb F$-algebra generated by $y_1, x_1,\ldots, y_n,x_n$ subject to the relations
\begin{equation}\label{WE}
\begin{aligned}
y_jy_i&=\widetilde{\lambda}_{ji}y_iy_j&& 1\leq i<j\leq n\\
y_jx_i&=\widetilde{\lambda}_{ij}x_iy_j&&1\leq i<j\leq n\\
x_jy_i&=\widetilde{q}_i\widetilde{\lambda}_{ij}y_ix_j&&1\leq i<j\leq n\\
x_jx_i&=\widetilde{q}_{i}^{-1}\widetilde{\lambda}_{ij}^{-1}x_ix_j&&1\leq i<j\leq n\\
x_iy_i-\widetilde{q}_iy_ix_i&=(\widetilde{q}_i-1)\left(1+\sum_{k=1}^{i-1}y_kx_k\right)&\qquad&1\leq  i\leq n.
\end{aligned}
\end{equation}

Write $A$ for $A_n^{\widetilde{Q},\widetilde{\Lambda}}$ for simplicity.
\end{defn}

\begin{lem}\label{SKP}
The algebra  $A$ is the iterated skew polynomial algebra
$$A=\Bbb F[y_1][x_1;\beta^1,\delta^1][y_2;\alpha^2][x_2;\beta^2,\delta^2]\ldots[y_n;\alpha^n][x_n;\beta^n,\delta^n],$$
where
\begin{equation}\label{WF}
\begin{aligned}
\alpha^j(y_i)&=\widetilde{\lambda}_{ji}y_i&&i<j,&\qquad \alpha^j(x_i)&=\widetilde{\lambda}_{ij}x_i&&i<j\\
\beta^j(y_i)&=\widetilde{q}_i\widetilde{\lambda}_{ij}y_i&&i\leq j,&\qquad \beta^j(x_i)&={\widetilde{q}_i}^{-1}{\widetilde{\lambda}_{ij}}^{-1}x_i&&i<j\\
\delta^j(y_i)&=0&&i<j, &\qquad\delta^j(x_i)&=0&&i<j \\
\delta^j(y_j)&=(\widetilde{q}_j-1)\left(1+\sum_{k=1}^{j-1}y_kx_k\right)&&1\leq j\leq n.&&&&
\end{aligned}
\end{equation}
Hence $\{y_1^{r_1}x_1^{s_1}\cdots y_n^{r_n}x_n^{s_n}|r_i,s_i=0,1,\ldots\}$ forms an $\Bbb F$-basis of $A$.
\end{lem}

\begin{proof}
It is verified by (\ref{WE}).
\end{proof}

\begin{lem}\label{par}
Let ${\bf  K}$ be the set  consisting of $\lambda\in\Bbb C\setminus\{0,1\}$ such that $e_i(\lambda)\neq0$ for all $1\leq i\leq r$ and the multiplicative subgroup
$\langle e_1(\lambda),\ldots, e_r(\lambda)\rangle$ is a free abelian group with basis $e_1(\lambda),\ldots, e_r(\lambda)$.
Then ${\bf K}$ is  an infinite subset of $\Bbb C\setminus\{0,1\}$ and $q\in{\bf K}$.
\end{lem}

\begin{proof}
Let $f_1(t),\ldots, f_r(t)$ be nonconstant and continuous  functions on a non-empty open set $U\subseteq\Bbb C\setminus\{0,1\}$ such that $f_i(t)/f_j(t)$ is nonconstant for each pair $(i,j)$ with $i\neq j$.
We show that there are uncountably many $0,1\neq\lambda\in\Bbb C$ such that  $\{f_i(\lambda)\}$ is $\Bbb Q$-linearly independent. We proceed by using induction on $r$. If $r=1$ then it is true clearly since $f_1$ is a nonconstant and continuous function.  Suppose that $r>1$ and that the claim is true for $r-1$. By the induction hypothesis, there are uncountably many $0,1\neq\lambda\in\Bbb C$ such that  $\sum_{i=1}^{r-1}a_i(f_i/f_r)(\lambda)\neq0$  for any $0\neq(a_1,\ldots,a_{r-1})\in\Bbb Q^{r-1}$. Hence, for any $a_r\in\Bbb Q$ such that $0\neq(a_1,\ldots,a_{r})$,  $\sum_{i=1}^{r-1} a_i(f_i/f_r)(\lambda)+a_r\neq0$ for uncountably many $0,1\neq\lambda\in \Bbb C$ since the cardinality of $\Bbb Q$ is countable. It follows that  $\sum_{i=1}^{r}a_if_i(\lambda)\neq0$ for any $0\neq(a_1,\ldots,a_r)\in\Bbb Q^r$ and thus the claim is verified.

Let $\ell$ be a half line starting from the origin such that $e_i(q)\notin\ell$ for all $i=1,\ldots,r$. Set $f_i(t)=\log e_i(t)$ for $i=1,\ldots,r$, where the function $\log$ is defined in the branch determined by $\ell$. Then all $f_i(t)$ are nonconstant and continuous on a nonempty open set $q\in U\subseteq\Bbb C\setminus\{0,1\}$
such that each $e_i(U)$ is contained in the given branch
since  $e_i(t)$ is a nonconstant polynomial by (\ref{WC}). For each pair $(i,j)$ with $i\neq j$,  suppose that $f_i(t)/f_j(t)=c\in\Bbb Q$. Then $\eta_i=e_i(q)=e_j^c(q)=\eta_j^c$, which is impossible since $\{\eta_k\}_{k=1}^r$ is a basis of $G(Q,\Lambda)$. If $f_i(t)/f_j(t)=c\in\Bbb C\setminus\Bbb Q$ then $e_i=e_j^c$, a contradiction since all $e_i$ are polynomials. Hence, for each pair $(i,j)$ with $i\neq j$,  $f_i(t)/f_j(t)$ is not constant.
 Therefore  there are uncountably many $0,1\neq\lambda\in\Bbb C$ such that  $\{f_i(\lambda)\}$ is $\Bbb Q$-linearly independent by the claim. This completes the proof.
\end{proof}

We apply Notation~\ref{ASSUM} to the triple $({\bf K},\Bbb F, A)$. Note that $({\bf K},\Bbb F, A)$ satisfies Notation~\ref{ASSUM}(1)-(3).

\begin{notn}
Note that $t-1$ is a nonzero, nonunit and non-zero-divisor of $A$.
Set $$A_1=A/(t-1)A.$$

It is proved by the proof of Theorem~\ref{ZZX}(3) that $A_1$ is a commutative $\Bbb C$-algebra. Thus $({\bf K},\Bbb F, A)$ satisfies the condition of Notation~\ref{ASSUM}(4).
\end{notn}

For each $\lambda\in{\bf K}$, the $\Bbb C$-algebra $A_\lambda$
is defined by Notation~\ref{ASSUM}(5).

\begin{prop}\label{ZZW}
(1) For each $\lambda\in{\bf K}$,  $A_{\lambda}$ is the $\Bbb C$-algebra
 generated by $y_1,x_1,\ldots, y_n,x_n$ subject to the relations
\begin{equation}\label{WFF}
\begin{aligned}
y_jy_i&=\widetilde{\lambda}_{ji}(\lambda)y_iy_j&& 1\leq i<j\leq n\\
y_jx_i&=\widetilde{\lambda}_{ij}(\lambda)x_iy_j&&1\leq i<j\leq n\\
x_jy_i&=\widetilde{q}_i(\lambda)\widetilde{\lambda}_{ij}(\lambda)y_ix_j&&1\leq i<j\leq n\\
x_jx_i&=\widetilde{q}_{i}^{-1}(\lambda)\widetilde{\lambda}_{ij}^{-1}(\lambda)x_ix_j&&1\leq i<j\leq n\\
x_iy_i-\widetilde{q}_i(\lambda)y_ix_i&=(\widetilde{q}_i(\lambda)-1)\left(1+\sum_{k=1}^{i-1}y_kx_k\right)&\qquad&1\leq  i\leq n
\end{aligned}
\end{equation}
and an iterated skew polynomial algebra
$$A_\lambda=\Bbb C[y_1][x_1;\beta_\lambda^1,\delta_\lambda^1][y_2;\alpha_\lambda^2][x_2;\beta_\lambda^2,\delta_\lambda^2]\cdots[y_n;\alpha_\lambda^n][x_n;\beta_\lambda^n,\delta_\lambda^n]$$ for suitable $\alpha_\lambda^i,\beta_\lambda^i,\delta_\lambda^i$. Hence $\{y_1^{r_1}x_1^{s_1}\cdots y_n^{r_n}x_n^{s_n}|r_i,s_i=0,1,\ldots\}$ forms a $\Bbb C$-basis of $A_\lambda$.

(2) In particular,
 $A_{q}$ is
the $\Bbb C$-algebra generated by $y_1,x_1,\ldots, y_n,x_n$ subject to the relations
\begin{equation}\label{WG}
\begin{aligned}
y_jy_i&={\lambda}_{ji}y_iy_j&& 1\leq i<j\leq n\\
y_jx_i&={\lambda}_{ij}x_iy_j&&1\leq i<j\leq n\\
x_jy_i&={q}_i{\lambda}_{ij}y_ix_j&&1\leq i<j\leq n\\
x_jx_i&={q}_{i}^{-1}{\lambda}_{ij}^{-1}x_ix_j&&1\leq i<j\leq n\\
x_iy_i-{q}_iy_ix_i&=({q}_i-1)\left(1+\sum_{k=1}^{i-1}y_kx_k\right)&\qquad&1\leq  i\leq n.
\end{aligned}
\end{equation}
\end{prop}

\begin{proof}
(1) It follows from (\ref{WE}) and (\ref{WF}).

(2)
Since $q\in{\bf K}$ and $e_i(q)=\eta_i$ for all $i$,
$$\widetilde{\lambda}_{ij}(q)=\lambda_{ij},\ \ \widetilde{q}_{i}(q)=q_{i} $$
for all $i,j$. Hence the result follows from (\ref{WFF}) by replacing $\lambda$ by $q$.
\end{proof}

\begin{lem}\label{IS}
The multi-parameter quantized Weyl algebra $R_n^{Q,\Lambda}$ is isomorphic to $A_{q}$.
\end{lem}

\begin{proof}
Define a homomorphism of $\Bbb C$-algebras
\begin{equation}\label{iso}
\psi:R_n^{Q,\Lambda}\longrightarrow A_{q}  \ \ \psi(x_i)=x_i,\ \psi(y_i)=(q_i-1)^{-1}y_i
 \end{equation}
 for all $i$. Then $\psi$ is an isomorphism by (\ref{WA}) and (\ref{WG}).
\end{proof}

\begin{thm}\label{ZZX}
Let $G(Q,\Lambda)$ be torsion free.

(1) $\widetilde{Q}'(1):=(\widetilde{q}_1'(1),\ldots, \widetilde{q}_n'(1))$ is an $n$-tuple of elements of $\Bbb C\setminus\{0\}$.

(2) $\widetilde{\Lambda}'(1):=\left(\widetilde{\lambda}_{ij}'(1)\right)$ is a skew-symmetric $n\times n$-matrix over $\Bbb C$.

(3)   The $\Bbb C$-algebra $A_1$ is  the
Poisson algebra $\Bbb C[y_1,x_1,\ldots, y_n,x_n]$ with Poisson bracket
\begin{equation}\label{WH}
\begin{aligned}
\{y_j,y_i\}&=\widetilde{\lambda}_{ji}'(1)y_iy_j&& 1\leq i<j\leq n\\
\{y_j,x_i\}&=\widetilde{\lambda}_{ij}'(1)x_iy_j&&1\leq i<j\leq n\\
\{x_j,y_i\}&=\left(\widetilde{q}_{i}'(1)+\widetilde{\lambda}_{ij}'(1)\right)y_ix_j&&1\leq i<j\leq n\\
\{x_j,x_i\}&=-\left(\widetilde{q}_{i}'(1)+\widetilde{\lambda}_{ij}'(1)\right)x_ix_j&&1\leq i<j\leq n\\
\{x_i,y_i\}&=\widetilde{q}_i'(1)\left(1+\sum_{k=1}^{i}y_kx_k\right)&\qquad&1\leq i\leq n,
\end{aligned}
\end{equation}
where $\widetilde{q}_i'$, $\widetilde{\lambda}_{ij}'$ are the  formal derivatives of  $\widetilde{q}_i$, $\widetilde{\lambda}_{ij}$ respectively.
We will call $A_1$ {\it the  multi-parameter Poisson Weyl algebra}.
\end{thm}

\begin{rem}
Note that $({\bf K},\Bbb F, A)$ satisfies the condition of Notation~\ref{ASSUM}(7) by Lemma~\ref{SKP}, Proposition~\ref{ZZW}(1) and Theorem~\ref{ZZX}(3).
We retain the notations in Notation~\ref{ASSUM} for $({\bf K},\Bbb F,A)$.
\end{rem}

\begin{proof}[Proof of Theorem~\ref{ZZX}]
(1) Let $q_i=\eta_1^{s_{1i}}\eta_2^{s_{2i}}\cdots\eta_r^{s_{ri}}$.
Since $q_i\neq1$,  $s_{ki}\neq0$ for some $k$. Thus, by (\ref{WC}),
$$\widetilde{q_i}'(1)=s_{1i}\mu_1+\ldots+s_{ri}\mu_r\neq0.$$

(2) For all $1\leq i,j\leq n$, since $\lambda_{ij}\lambda_{ji}=1$, $\widetilde{\lambda}_{ij}\widetilde{\lambda}_{ji}=1$ by (\ref{WB}) and thus $\widetilde{\lambda}_{ij}'(1)+\widetilde{\lambda}_{ji}'(1)=0$ by (\ref{WC}).

(3) Since $e_i(1)=1$ by (\ref{WC}),  $\widetilde{q}_i(1)=1, \widetilde{\lambda}_{ij}(1)=1$ for all $i,j$ and thus $A_1$ is commutative by (\ref{WE}). Let us find the Poisson bracket in $A_1$. For $i<j$,
$$\{y_j,y_i\}=(t-1)^{-1}(y_jy_i-y_iy_j)|_{t=1}
=\left(\frac{\widetilde{\lambda}_{ji}-1}{t-1}\right)(1)y_iy_j=\widetilde{\lambda}_{ji}'(1)y_iy_j
$$
and thus the first formula of (\ref{WH}) is obtained. The others of (\ref{WH}) are obtained similarly.
\end{proof}

\begin{notn}
In $A_{q}$ and $A_1$, set
$$z_0=1,\ \ \ z_i=1+\sum_{k=1}^iy_kx_k$$
for $i=1,\ldots,n$.
\end{notn}

\begin{lem}\label{BR}
(1) In $A_{q}$,
$$x_iy_i-q_iy_ix_i=(q_i-1)z_{i-1},\ \ \ x_iy_i-y_ix_i=(q_i-1)z_i,  \ \ \ 1\leq i\leq n$$
and
$$\begin{array}{llll}
y_j z_i=z_iy_j&  i<j;  &x_j z_i=z_ix_j& i<j;\\
y_j z_i=q_j^{-1}z_iy_j& i\geq j;  &x_j z_i=q_jz_ix_j& i\geq j;\\
z_iz_j=z_jz_i&\text{all } i,j.&&
\end{array}$$

(2) In $A_1$,
$$\{x_i,y_i\}-\widetilde{q}_i'(1)y_ix_i=\widetilde{q}_i'(1)z_{i-1},\ \ \ \{x_i, y_i\}=\widetilde{q}_i'(1)z_i,  \ \ \ 1\leq i\leq n$$
and
$$\begin{array}{llll}
\{y_j, z_i\}=0& i<j;  &\{x_j, z_i\}=0& i<j;\\
\{y_j, z_i\}=-\widetilde{q}_j'(1)y_jz_i& i\geq j;  &\{x_j, z_i\}=\widetilde{q}_j'(1)x_jz_i& i\geq j;\\
\{z_i, z_j\}=0&\text{all } i,j.&&
\end{array}$$
\end{lem}

\begin{proof}
These are checked routinely by (\ref{WG}) and (\ref{WH}).
\end{proof}

\begin{defn}
A subset $T$ of $\mathcal{M}_n=\{z_1,z_2, y_2, x_2, z_3,y_3,x_3,\ldots, z_n, y_n,x_n\}$ is said to be {\it admissible} if $T$ satisfies the condition:
For $2\leq i\leq n$,
$$y_i\in T\text{ or } x_i\in T\Leftrightarrow z_i\in T\text{ and }z_{i-1}\in T.$$

Note that if $T$ is an admissible set  of $A_q$ then the ideal $\langle T\rangle$ generated by $T$ is equal to the right (left) ideal generated by $T$ since each $z_i$ is normal and
$y_i$ and $x_i$ are normal modulo $\langle z_{i-1}\rangle$ by Lemma~\ref{BR}(1).
\end{defn}

\begin{lem}\label{AP}
(1) Let $P$ be a prime ideal of $A_{q}$. Then $P\cap\mathcal{M}_n$ is an admissible set.

(2) Let $P$ be a Poisson prime ideal of $A_1$. Then $P\cap\mathcal{M}_n$ is an admissible set.
\end{lem}

\begin{proof}
(1)  If
$y_i\in P\cap\mathcal{M}_n$ or $x_i\in P\cap\mathcal{M}_n$ then $z_i, z_{i-1}\in P\cap\mathcal{M}_n$ by Lemma~\ref{BR}(1). Conversely, if $z_i, z_{i-1}\in P\cap\mathcal{M}_n$
then $y_ix_i=z_i-z_{i-1}\in P$. Since $y_i$ and $x_i$ are normal modulo  $\langle z_{i-1}\rangle$, $y_i\in P\cap\mathcal{M}_n$ or $x_i\in P\cap\mathcal{M}_n$.

(2) Similar to (1).
\end{proof}

\begin{cor}\label{DIS}
For each admissible set $T$, set
$$\begin{aligned}
\spec_T A_{q}&=\{P\in\spec A_q|P\cap\mathcal{M}_n=T\},\\
 \Pspec_T A_1&=\{P\in\Pspec A_1|P\cap\mathcal{M}_n=T\}.
 \end{aligned}
$$
Then
 $$\spec A_{q}=\biguplus_T\spec_T A_{q},\ \ \ \Pspec A_1=\biguplus_T\Pspec_T A_1.$$
\end{cor}

\begin{proof}
It is clear by Lemma~\ref{AP}
\end{proof}

\begin{rem}\label{MR}
 Recall the following concepts in \cite[4.2]{AkJo}. In $R_n^{Q,\Lambda}$, set $z_i=1+\sum_{k=1}^i(q_k-1)y_kx_k$ for each $1\leq i\leq n$. For a subset $U$ of $\{z_1,y_1,x_1,z_2,y_2,x_2,\ldots, z_n,y_n,x_n\}\subseteq R_n^{Q,\Lambda}$ and an integer
$i$ with $1\leq i\leq n$, we say that $i$ features in $U$ if at least one of $z_i,y_i,x_i$ is in $U$. By $p$-sequence, we mean a subset $U$ of $\{z_1,y_1,x_1,\ldots, z_n,y_n,x_n\}$ with the following properties:

(i)  if $z_i\in U$ then $i-1$ does not feature in $U$;

(ii) if $y_i\in U$ or $x_i\in U$ then $i-1$ does feature in $U$.

Let $U$ be a $p$-sequence such that $y_1,x_1\notin U$ and set $U'=U\cup\{z_i|\text{$i$ features in $U$}\}$. Then the homomorphic image $\psi(U')$, where $\psi$ is the isomorphism  (\ref{iso}), consists of elements obtained by multiplying  nonzero scalars to all elements in an admissible set $T$. Thus the ideal generated by $\psi(U)$ is equal to the ideal $\langle T\rangle$. Conversely, for any admissible set $T$ of $A_{q}$,  $\langle T\rangle$ is equal to an ideal generated by
$\psi(U)$ for some $p$-sequence $U$  such that $y_1,x_1\notin U$.
\end{rem}

\begin{lem}
Let $G(Q,\Lambda)$ be  torsion free and let  $T$ be an admissible set.

(1) The ideal $\langle T\rangle$ of $A_{q}$ is a completely prime ideal.

(2) The ideal $\langle T\rangle$ of $A_1$ is a  Poisson prime ideal.
\end{lem}

\begin{proof}
(1) It follows immediately by Remark~\ref{MR} and  \cite[4.5]{AkJo}.

(2)
We proceed by induction on $n$. If $n=1$ then the admissible sets of $\mathcal{M}_1$ are only $\varnothing$ and $\{z_1\}$.
If $T=\varnothing$ then $\langle T\rangle=0$ and thus the result is  clear. Let  $T=\{z_1\}$. Since  $z_1$ is a Poisson normal element,  $\langle T\rangle$
is Poisson prime.

Suppose that $n>1$ and that the result is true for the cases less than $n$. Set $U=T\cap\mathcal{M}_{n-1}$. Then $U$ is an admissible set of $\mathcal{M}_{n-1}$. There are following five cases:
$$T=U, T=U\cup\{z_n\}, T=U\cup\{z_n, y_n\}, T=U\cup\{z_n, x_n\}, T=U\cup\{z_n, y_n,x_n\}.$$
If $T=U$ then the result is clear by the induction hypothesis. For the other cases, $\langle T\rangle$ is Poisson prime since $z_n$ is Poisson normal and
$y_n$ and $x_n$ are Poisson normal modulo $\langle z_n\rangle$.
\end{proof}

\begin{notn}
Let $T$ be an admissible set of $\mathcal{M}_n$.
 Define a subset $Y_T$ of $\mathcal{M}_n\cup\{y_1\}$ by
$$Y_T\cap\{z_i,y_i,x_i\}=\left\{\begin{aligned}
&\{z_i,y_i\}& &\text{if }z_i\notin T,\\
&\{y_i\}&&\text{if }z_i\in T,y_i\notin T,\\
&\{x_i\}&&\text{if }z_i\in T,y_i\in T, x_i\notin T,\\
&\varnothing&&\text{if }z_i\in T,y_i\in T, x_i\in T\end{aligned}\right.$$
for each $i=1,\ldots,n$. Note that there is no $i$ such that $y_i\in Y_T$ and $x_i\in Y_T$.
\end{notn}

Recall the  coordinate ring of  quantum torus: Let $C=(c_{ij})$ be a multiplicative antisymmetric $n\times n$-matrix over $\Bbb C$.
The  coordinate ring  of quantum torus, denoted by  $\mathcal{T}_C$,  is the $\Bbb C$-algebra generated by $z_1^{\pm1},\ldots, z_n^{\pm1}$ subject to the relations
$$z_iz_j=c_{ij}z_jz_i$$ for all $1\leq i,j\leq n$.

The proof of the following lemma is modified from that of \cite[4.9]{AkJo} by using Poisson terminologies.

\begin{lem}\label{TOR}
Let $T$ be an admissible set of $\mathcal{M}_n$ and set
$Y_T=\{w_1,\ldots, w_s\}.$  Denote by $\mathcal{Y}_T$  the multiplicative closed set of $A_{q}/\langle T\rangle$ and $A_1/\langle T\rangle$ generated by the canonical image of $Y_T$.

(1) Let $\Bbb C[w_1,\ldots, w_s]$ be the subalgebra of $A_{q}$ generated by $w_1,\ldots, w_s$. Then, for all $1\leq i,j\leq s$,
$$w_iw_j=c_{ij}w_jw_i$$ for some $c_{ij}\in G(Q,\Lambda)$ and  $(A_{q}/\langle T\rangle)[\mathcal{Y}_T^{-1}]$ is isomorphic to the  localization $\Bbb C[w_1^{\pm1},\ldots, w_s^{\pm1}]$ of $\Bbb C[w_1,\ldots, w_s]$
at the multiplicative closed set generated by $w_1,\ldots, w_s$, which is the  coordinate ring  of quantum torus $\mathcal{T}_{C_T}$, where ${C_T}=(c_{ij})$ a multiplicative anti-symmetric $s\times s$-matrix.

(2) Let $\Bbb C[w_1,\ldots, w_s]$ be the Poisson subalgebra of $A_1$ generated by $w_1,\ldots, w_s$. Then, for all $1\leq i,j\leq s$,
$$\{w_i,w_j\}=\widetilde{c}_{ij}'(1)w_iw_j,$$
  where $c_{ij}$ is the element of $G(Q,\Lambda)$ in (1), and $(A_1/\langle T\rangle)[\mathcal{Y}_T^{-1}]$ is Poisson isomorphic to the  localization $\Bbb C[w_1^{\pm1},\ldots, w_s^{\pm1}]$ of $\Bbb C[w_1,\ldots, w_s]$
at the multiplicative closed set generated by $w_1,\ldots, w_s$.
\end{lem}

\begin{proof}
(1) There is no $i$ such that $y_i\in Y_T$ and $x_i\in Y_T$. Hence, for all $1\leq i,j\leq s$, $w_iw_j=c_{ij}w_jw_i$ for a suitable element $c_{ij}\in G(Q,\Lambda)$ by (\ref{WG}) and Lemma~\ref{BR}(1). The remaining result is proved by \cite[4.8 and the proof of 4.9]{AkJo} since  $Y_T$ consists of the elements obtained by multiplying nonzero scalars to the elements of $\psi(D(U))$ for some $p$-sequence $U$ by Remark~\ref{MR}, where $\psi$ is the isomorphism  (\ref{iso}) and $D(U)$ is the subset of $R_n^{Q,\Lambda}$ in \cite[4.8]{AkJo}.

(2) Note that  the canonical map from the subalgebra of $A_1$ generated by $Y_T$ to $A_1/\langle T\rangle$ is injective and that, for $i$ with $z_i\notin T$, $y_i$ and $z_i$ are in $Y_T$
and $\overline{x}_i=\overline{y}_i^{^{-1}}\overline{(z_i-z_{i-1})}$  in  $(A_1/\langle T\rangle)[\mathcal{Y}_T^{-1}]$. Hence $(A_1/\langle T\rangle)[\mathcal{Y}_T^{-1}]$ is the Poisson algebra $\Bbb C[w_1^{\pm1},\ldots, w_s^{\pm1}]$.

It remains to  prove the Poisson bracket $\{w_i,w_j\}=\widetilde{c}_{ij}'(1)w_iw_j$ for all $1\leq i,j\leq s$. We proceed by induction on $n$. If $n=1$ then
$Y_T=\{z_1,y_1\}$ or $Y_T=\{y_1\}$.
If $Y_T=\{z_1,y_1\}$ then $z_1y_1=q_1y_1z_1$ in $A_{q}$ and $\{z_1,y_1\}=\widetilde{q}_1'(1)z_1y_1$ in $A_1$ by Lemma~\ref{BR}. If $Y_T=\{y_1\}$ then $y_1y_1=y_1y_1$ in $A_{q}$ and $\{y_1,y_1\}=0$ in $A_1$. Hence our claim is true for $n=1$.

Suppose that $n>1$ and that our claim is true  for the cases less than $n$. Let $W=T\cap\mathcal{M}_{n-1}$. Then $W$ is an admissible set of $\mathcal{M}_{n-1}$. Set $Y_W=\{w_1,\ldots, w_p\}$. Let $\Delta$ be the multiplicative anti-symmetric $p\times p$-matrix $(c_{ij})$ which is determined by the commutation relations $w_iw_j=c_{ij}w_jw_i$ for all $1\leq i,j\leq p$ and let $\Delta'$ be the skew symmetric $p\times p$-matrix $(d_{ij})$ which is determined by the Poisson brackets $\{w_i,w_j\}=d_{ij}w_iw_j$ for all $1\leq i,j\leq p$. By  induction hypothesis, $d_{ij}=\widetilde{c}_{ij}'(1)$ for all $1\leq i,j\leq p$.
 There are the following five cases:
$$T=W, T=W\cup\{z_n\}, T=W\cup\{z_n, y_n\}, T=W\cup\{z_n, x_n\}, T=W\cup\{z_n, y_n, x_n\}.$$

If $T=W$ then $Y_T=Y_W\cup\{z_n,y_n\}=\{w_1,\ldots, w_p,z_n,y_n\}$. In $A_{q}$, the commutation relations of each pair of elements in $Y_T$ are determined by the  multiplicative anti-symmetric matrix:
$$\left(\begin{matrix}
\Delta&U\\U^*&V
\end{matrix}\right),$$
where  $U$ is the $p\times 2$-matrix determined by the commutation relations between $\{w_1,\ldots, w_p\}$ and $\{z_n,y_n\}$,
each $(k,\ell)$-entry of $U^*$ is the inverse of the $(\ell,k)$-entry of $U$ and $V$ is the multiplicative anti-symmetric matrix $\left(\begin{matrix}1&q_n\\q_n^{-1}&1\end{matrix}\right)$ determined by the commutation relations of $\{z_n, y_n\}$. In $A_1$, the Poisson brackets of each pair of elements in $Y_T$ are determined by the  skew symmetric matrix:
$$\left(\begin{matrix}
\Delta'&U'\\-(U')^t&V'
\end{matrix}\right),$$
where  $U'$ is the $p\times 2$-matrix determined by the Poisson brackets between $\{w_1,\ldots, w_p\}$ and $\{z_n,y_n\}$,
$-(U')^t$ is the transpose of $U'$ and $V'$ is the skew symmetric matrix $\left(\begin{matrix}0&\widetilde{q}_n'(1)\\-\widetilde{q}_n'(1)&0\end{matrix}\right)$
determined by  $\{z_n, y_n\}$.
  Observe that each entry of $V'$ is obtained from the corresponding entry of $V$ by derivation at $t=1$. Let us check that each $(i,j)$-entry of $U'$ is obtained from the $(i,j)$-entry of $U$ by derivation at $t=1$. Let $U_i$ and $U'_i$ be the $i$th-rows of $U$ and $U'$, respectively,  for $1\leq i\leq p$. If $w_1=z_1$ then $U_1=(1,1)$ and $U'_1=(0,0)$.  If $w_1=y_1$ then  $U_1=(q_1^{-1},\lambda_{1n})$ and $U'_1=(-\widetilde{q_1}'(1),\widetilde{\lambda}_{1n}'(1))$. Hence our claim is true for $U_1$ and $U'_1$. Suppose that, for $i=1,\ldots, p-1$, each entry of  $U'_i$ is obtained from the corresponding entry of $U_i$ by derivation at 1.
If $w_p=z_k$ for some $k<n$ then
$$U_p=(1, 1),\ \ \ U'_p=(0, 0).$$
If $w_p=y_k$ for some $k<n$ then
$$U_p=(q_k^{-1}, \lambda_{kn}),\ \ \ U'_p=(-\widetilde{q}_k'(1), \widetilde{\lambda}_{kn}'(1)).$$
If $w_p=x_k$ for some $k<n$ then
$$U_p=(q_k, \lambda_{kn}^{-1}),\ \ \ U'_p=(\widetilde{q}_k'(1), -\widetilde{\lambda}_{kn}'(1)).$$
Hence each $(i,j)$-entry of $U'$ is obtained from the $(i,j)$-entry of $U$ by derivation at $t=1$.

If $T=W\cup\{z_n\}$ then $Y_T=Y_W\cup\{y_n\}=\{w_1,\ldots, w_p,y_n\}$. In $A_{q}$, the commutation relations of each pair of elements in $Y_T$ are determined by the  multiplicative anti-symmetric matrix:
$$\left(\begin{matrix}
\Delta&U\\U^*&1
\end{matrix}\right),$$
where  $U$ is the $p\times 1$-matrix determined by the commutation relations between $\{w_1,\ldots, w_p\}$ and $\{y_n\}$ and each $(1,\ell)$-entry of $U^*$ is the inverse of the $(\ell,1)$-entry of $U$.
In $A_1$, the Poisson brackets of each pair of elements in $Y_T$ are determined by the  skew symmetric matrix:
$$\left(\begin{matrix}
\Delta'&U'\\-(U')^t&0
\end{matrix}\right),$$
where  $U'$ is the $p\times 1$-matrix determined by the Poisson brackets between $\{w_1,\ldots, w_p\}$ and $\{y_n\}$ and
$-(U')^t$ is the transpose of $U'$.  Let $c_\ell$ and $d_\ell$ be $(\ell,1)$-entries of $U$ and $U'$, respectively. If $w_\ell=z_k$ for some $k<n$ then
$c_\ell=1$, $d_\ell=0$.  If $w_\ell=y_k$ for some $k<n$ then
$c_\ell=\lambda_{kn}$, $d_\ell=\widetilde{\lambda}_{kn}'(1)$.  If $w_\ell=x_k$ for some $k<n$ then
$c_\ell=\lambda_{kn}^{-1}$, $d_\ell=-\widetilde{\lambda}_{kn}'(1)$. Hence our claim is true.

If $T=W\cup\{z_n,y_n\}$ then $Y_T=Y_W\cup\{x_n\}=\{w_1,\ldots, w_p,x_n\}$. In $A_{q}$, the commutation relations of each pair of elements in $Y_T$ are determined by the  multiplicative anti-symmetric matrix:
$$\left(\begin{matrix}
\Delta&U\\U^*&1
\end{matrix}\right),$$
where  $U$ is the $p\times 1$-matrix determined by the commutation relations between $\{w_1,\ldots, w_p\}$ and $\{x_n\}$ and each $(1,\ell)$-entry of $U^*$ is the inverse of the $(\ell,1)$-entry of $U$.
In $A_1$, the Poisson brackets of each pair of elements in $Y_T$ are determined by the  skew symmetric matrix:
$$\left(\begin{matrix}
\Delta'&U'\\-(U')^t&0
\end{matrix}\right),$$
where  $U'$ is $p\times 1$-matrix determined by the Poisson brackets between $\{w_1,\ldots, w_p\}$ and $\{x_n\}$ and
$-(U')^t$ is the transpose of $U'$.  Let $c_\ell$ and $d_\ell$ be $(\ell,1)$-entries of $U$ and $U'$, respectively. If $w_\ell=z_k$ for some $k<n$ then
$c_\ell=1$, $d_\ell=0$.  If $w_\ell=y_k$ for some $k<n$ then
$c_\ell=q_k^{-1}\lambda_{kn}^{-1}$, $d_\ell=-(\widetilde{q}_k'(1)+\widetilde{\lambda}_{kn}'(1))$.  If $w_\ell=x_k$ for some $k<n$ then
$c_\ell=q_k\lambda_{kn}$, $d_\ell=\widetilde{q}_k'(1)+\widetilde{\lambda}_{kn}'(1)$.

If $T=W\cup\{z_n,x_n\}$ then $Y_T=Y_W\cup\{y_n\}=\{w_1,\ldots, w_p,y_n\}$. This case was proved in the case $T=W\cup\{z_n\}$.

If $T=W\cup\{z_n,y_n,x_n\}$ then $Y_T=Y_W$ and thus there is nothing to prove. This completes the proof.
\end{proof}

\begin{lem}\label{CENT}
 Retain the notation of Lemma~\ref{TOR}.
Let $$c_{\ell i}=\eta_1^{v_{\ell i}^1}\cdots \eta_r^{v_{\ell i}^r}=e_1(q)^{v_{\ell i}^1}\cdots e_r(q)^{v_{\ell i}^r}$$ for all $1\leq\ell, i\leq s$.

(1) The center $Z((A_q/\langle T\rangle)[\mathcal{Y}_T^{-1}])$ is  the $\Bbb C$-subalgebra generated by monomials $w_1^{u_1}\cdots w_s^{u_s}$ such that $u_1v_{\ell 1}^i+\ldots +u_sv_{\ell s}^i=0$ for all $1\leq i\leq r$.
In particular, $Z((A_q/\langle T\rangle)[\mathcal{Y}_T^{-1}])$ is independent to $q$.

(2) The Poisson center $ZP((A_1/\langle T\rangle)[\mathcal{Y}_T^{-1}])$ is  the $\Bbb C$-subalgebra generated by monomials $w_1^{u_1}\cdots w_s^{u_s}$ such that $u_1v_{\ell 1}^i+\ldots +u_sv_{\ell s}^i=0$ for all $1\leq i\leq r$.
\end{lem}

\begin{proof}
For $u=(u_1,\ldots, u_s)\in\Bbb Z^s$, set $w^u=w_1^{u_1}\cdots w_s^{u_s}$.

(1) For $f=\sum_u \alpha_u w^u\in (A_q/\langle T\rangle)[\mathcal{Y}_T^{-1}]$,
$f\in Z((A_q/\langle T\rangle)[\mathcal{Y}_T^{-1}])$ if and only if $w_\ell fw_\ell^{-1}=f$ for all $\ell$ if and only if $w^u\in Z((A_q/\langle T\rangle)[\mathcal{Y}_T^{-1}])$ for all $u$ such that $\alpha_u\neq0$, and
$$\begin{aligned}
w^u\in Z((A_q/\langle T\rangle)[\mathcal{Y}_T^{-1}])&\Leftrightarrow w_\ell w^uw_\ell^{-1}=w^u &&\text{for all }\ell=1,\ldots, s \\
&\Leftrightarrow c_{\ell 1}^{u_1}\cdots c_{\ell s}^{u_s}=1 &&\text{for all }\ell=1,\ldots,s\\
&\Leftrightarrow \prod_{i=1}^r\eta_i^{u_1v_{\ell 1}^i+\ldots +u_sv_{\ell s}^i}=1 &&\text{for all }\ell=1,\ldots,s\\
&\Leftrightarrow u_1v_{\ell 1}^i+\ldots +u_sv_{\ell s}^i=0 &&\text{for all }\ell\text{ and }  i.
\end{aligned}$$
Thus $Z((A_q/\langle T\rangle)[\mathcal{Y}_T^{-1}])$ is the $\Bbb C$-algebra spanned by all  monomials $w^u$ such that $$u_1v_{\ell 1}^i+\ldots +u_sv_{\ell s}^i=0$$ for all $\ell=1,\ldots,s$ and $i=1,\ldots,r$. That is,
$$Z((A_q/\langle T\rangle)[\mathcal{Y}_T^{-1}])=\left\{\sum_u\Bbb C w^u| u_1v_{\ell 1}^i+\ldots +u_sv_{\ell s}^i=0\ \text{ for all $1\leq\ell\leq s$,  $1\leq i\leq r$}\right\}.$$
In particular, $Z((A_q/\langle T\rangle)[\mathcal{Y}_T^{-1}])$ is independent to $q$.

(2) Note, by (\ref{WC}),  that $$\widetilde{c}_{\ell i}'(1)=v_{\ell i}^1\mu_1+\ldots+v_{\ell i}^r\mu_r.$$
For $f=\sum_u \alpha_u w^u\in (A_1/\langle T\rangle)[\mathcal{Y}_T^{-1}]$,
$f\in ZP((A_1/\langle T\rangle)[\mathcal{Y}_T^{-1}])$ if and only if $\{w_\ell, f\}w_\ell^{-1}=0$ for all $\ell$ if and only if $w^u\in ZP((A_1/\langle T\rangle)[\mathcal{Y}_T^{-1}])$ for all $u$ such that $\alpha_u\neq0$, and
 $$\begin{aligned}
w^u\in ZP((A_1/\langle T\rangle)[\mathcal{Y}_T^{-1}])&\Leftrightarrow \{w_\ell, w^u\}=0 &&\text{for all }\ell=1,\ldots,s \\
&\Leftrightarrow u_1\widetilde{c}_{\ell 1}'(1)+\ldots +u_s\widetilde{c}_{\ell s}'(1)=0 &&\text{for all }\ell=1,\ldots,s\\
&\Leftrightarrow\sum_{i=1}^r (u_1v_{\ell 1}^i+\ldots +u_sv_{\ell s}^i)\mu_i=0 &&\text{for all }\ell=1,\ldots,s\\
&\Leftrightarrow u_1v_{\ell 1}^i+\ldots +u_sv_{\ell s}^i=0 &&\text{for all }\ell\text{ and } i.
\end{aligned}$$
Hence
$$ZP((A_1/\langle T\rangle)[\mathcal{Y}_T^{-1}])=\left\{\sum_u\Bbb C w^u| u_1v_{\ell 1}^i+\ldots +u_sv_{\ell s}^i=0\ \text{ for all $1\leq\ell\leq s$,  $1\leq i\leq r$}\right\}.$$
 \end{proof}

\begin{lem}\label{CONT}
 Retain the notation of Lemma~\ref{TOR}.
Let $D$ be the $\Bbb C$-algebra $\Bbb C[w_1,\ldots, w_s]$ of Lemma~\ref{TOR}(1) and let $E$ be the Poisson algebra
$\Bbb C[w_1,\ldots, w_s]$ of Lemma~\ref{TOR}(2). Note that
$Z(D[w_1^{-1},\ldots, w_s^{-1}])=ZP(E[w_1^{-1},\ldots, w_s^{-1}])$ by Lemma~\ref{CENT}. Set
$$Z=Z(D[w_1^{-1},\ldots, w_s^{-1}])=ZP(E[w_1^{-1},\ldots, w_s^{-1}])$$ and
 $w^u=w_1^{u_1}\cdots w_s^{u_s}$ for $u=(u_1,\ldots, u_s)\in\Bbb Z^s$.
For any $f=\sum_u \alpha_u w^u\in Z$ and $v\in(\Bbb Z_+)^s$, $w^vf\in D$ if and only if $w^vf\in E$.
\end{lem}

\begin{proof}
In $D$, $w^vf=fw^v$ since $f\in Z$. Hence $w^vf\in D$ if and only if $v+u\in(\Bbb Z_+)^s$ for all $u$ such that $\alpha_u\neq0$
if and only if $w^vf\in E$.
\end{proof}

We recall \cite[Theorem 2.4 and preceding comment]{Oh4}: A Poisson $\Bbb C$-algebra $B$
is said to satisfy the {\it Dixmier-Moeglin equivalence} if the
following conditions are equivalent: For a Poisson prime ideal $P$
of $B$,
\begin{itemize}
\item[(i)] $P$ is Poisson primitive.
\item[(ii)] $P$ is rational (i.e., the Poisson center of the quotient field of $B/P$ is equal to $\Bbb C$).
\item[(iii)] $P$ is locally closed (i.e., the intersection of all Poisson prime ideals properly containing $P$
is strictly larger than $P$).
\end{itemize}

\begin{thm}\label{PRIM}
The multi-parameter Poisson Weyl algebra $A_1$ satisfies the Dixmier-Moeglin equivalence.
\end{thm}

\begin{proof}
Since $A_1$ is finitely generated, it is enough to show that every rational Poisson  prime ideal is locally closed by \cite[Proposition 1.10 and (ii)$\Rightarrow$(i) of the proof of Theorem 2.4]{Oh4}. Let $P$ be a rational Poisson prime ideal of $A_1$
and let $P\cap\mathcal{M}_n=T$. Then $T$ is an admissible set by Lemma~\ref{AP}.   Since the canonical image $P'$ of $P$ to $(A_1/\langle T\rangle)[\mathcal{Y}_T^{-1}]$ is rational, $P'$ is locally  closed by Lemma~\ref{TOR}(2) and \cite[Theorem 2.4]{Oh4}. It follows that $P$ is locally closed since admissible sets are finite.
\end{proof}

\begin{thm}\label{FIN}
Let $G(Q,\Lambda)$ be a torsion free group. Then there exists a homeomorphism  from $\spec R_n^{Q,\Lambda}$ onto $\Pspec A_1$ such that its restriction to $\prim R_n^{Q,\Lambda}$ is also a homeomorphism onto $\Pprim A_1$.
\end{thm}

\begin{proof}
It is enough to show the result for $A_q$ instead of $R_n^{Q,\Lambda}$ by Lemma~\ref{IS}. Let $T$ be an admissible set and set $Y_T=\{w_1,\ldots,w_s\}$. Denote by $\iota$ the canonical epimorphism from $A_q$ onto $A_q/\langle T\rangle$ or
the canonical epimorphism from $A_1$ onto $A_1/\langle T\rangle$. For any $u=(u_1,\ldots,u_s)\in\Bbb Z^s$, we set $w^u=w_1^{u_1}\cdots w_s^{u_s}$   and   $\iota(w^u)=\iota(w_1)^{u_1}\cdots \iota(w_s)^{u_s}$.
We may set $$Z=Z((A_q/\langle T\rangle)[\mathcal{Y}_T^{-1}])=ZP((A_q/\langle T\rangle)[\mathcal{Y}_T^{-1}])$$
by Lemma~\ref{CENT}.
For any prime ideal $P'$ of $Z$, set
\begin{equation}\label{WI}
J_{P'}=\left\{w^v(\sum_{u\in\Bbb Z^s,\alpha_u\neq0}\alpha_uw^u)|\sum\alpha_u\iota(w^u)\in P',v\in(\Bbb Z_+)^s \text{ such that }v+u\in(\Bbb Z_+)^s\ \forall u \right\}.
\end{equation}
Note that $\iota(J_{P'})$ is the contraction of $P'$ to  $A_q/\langle T\rangle$ and $A_1/\langle T\rangle$ by Lemma~\ref{CONT}
and that the ideal generated by $J_{P'}\cup T$ is a prime ideal of $A_q$ and a Poisson prime ideal of $A_1$ by \cite[1.5]{GoLet3} and \cite[Corollary 2.3]{Oh4}.

Let  $P\in \spec_T A_q$.  Then
 there exists a unique prime ideal $P'$ of $Z$ such that $\iota(P)$ is generated by the contraction of $P'$ to  $A_q/\langle T\rangle$  \cite[1.5]{GoLet3} and thus $P$ is generated by $J_{P'}\cup T$. Similarly, if $P\in \spec_T A_1$ then
 there exists a unique prime ideal $P'$ of $Z$ such that $\iota(P)$ is generated by the contraction of $P'$ to  $A_1/\langle T\rangle$  \cite[Corollary 2.3]{Oh4} and thus $P$ is generated by $J_{P'}\cup T$.

Let $\ \widehat{}\ :A_q\longrightarrow \widehat{A}$ be the map defined by $f(q)\mapsto\widehat{f(q)}=(f(\lambda))_{\lambda\in{\bf K}}$.
Since  all $\alpha_u$ in (\ref{WI}) are independent to $q$ by Lemma~\ref{CENT}(1), $\pi_\lambda(\widehat{x})=x$ and $\Gamma(\widehat{x})=x$ for $x\in J_{P'}\cup T$.
Denote by $X$ the  ideals of $\widehat{A}$ generated by $\widehat{J_{P'}\cup T}$ for all $P\in\spec_T A_q$ and admissible set $T$. Then the map $\ \widehat{}\ :A_q\longrightarrow \widehat{A}$ induces a homeomorphism
$\widehat{\varphi}$ from $\spec A_q$ onto $X$ and
$$\Gamma(\widehat{J_{P'}\cup T})=J_{P'}\cup T.$$
Hence the map $\varphi_\Gamma$ given in (\ref{GAMM}) is a homeomorphism from $X$ onto $\Pspec A_1$ by Corollary~\ref{INC}. It follows that $\spec A_q$ is homeomorphic to $\Pspec A_1$ by the composition $\varphi_\Gamma\circ\widehat{\varphi}$.
The restriction of this homeomorphism to the primitive spectrum $\prim A_q$ is also a homeomorphism onto the Poisson primitive spectrum $\Pprim A_1$ by \cite[9.4]{Good4}, \cite[5.6]{GoLet1} and Theorem~\ref{PRIM}.
\end{proof}

Let $\mathcal{O}(V)$ be the coordinate ring of a Poisson affine variety $V$. The sets $\{p\in V|\mathcal{P}(\frak m_p)=P\}$ for $P\in\Pprim\mathcal{O}(V)$ are called the {\it symplectic cores} in $V$, where $\frak m_p=\{f\in\mathcal{O}(V)|f(p)=0\}$. Refer to \cite[\S6]{Good4} for symplectic cores.

\begin{cor}
Let $G(Q,\Lambda)$ be a torsion free group. Then  $\prim R_n^{Q,\Lambda}$ is homeomorphic to the space of symplectic cores in the affine variety $\Bbb C^n$, where the Poisson bracket in $\Bbb C^n$ is given by $A_1$.
\end{cor}

\begin{proof}
It follows immediately by Theorem~\ref{PRIM}, Theorem~\ref{FIN} and \cite[Lemma 9.3]{Good4}.
\end{proof}


\bibliographystyle{amsplain}



\providecommand{\bysame}{\leavevmode\hbox to3em{\hrulefill}\thinspace}
\providecommand{\MR}{\relax\ifhmode\unskip\space\fi MR }
\providecommand{\MRhref}[2]{%
  \href{http://www.ams.org/mathscinet-getitem?mr=#1}{#2}
}
\providecommand{\href}[2]{#2}

\end{document}